\documentclass[12pt]{article}
\usepackage[english]{babel}
\usepackage{amsmath,amsthm,amssymb}
\usepackage[top=2.8cm, bottom=2.8cm, left=2.8cm, right=2.8cm]{geometry}
\usepackage{graphicx}
\usepackage{amsfonts}
\usepackage{appendix}

\newcommand{\winsequence}[1]{\underline{#1}}
\newcommand{\overmark}[1]{\overline{#1}} 

\theoremstyle{theorem}
\newtheorem{theorem}{Theorem}
\theoremstyle{definition}

\theoremstyle{corollary}
\newtheorem{cor}{Corollary}
\theoremstyle{lemma}
\newtheorem{lemma}{Lemma}

\begin{document}

\title{Catch-Up: A Rule That Makes Service Sports More Competitive\footnote{This paper is forthcoming in \textit{The American Mathematical Monthly}, 2018.}}
\markright{The Catch-Up Rule in Service Sports}
\author{\\ \\ Steven J. Brams\\
	\small{Department of Politics, New York University, New York, NY  10012, USA} \\ \small{steven.brams@nyu.edu}\\ \\
	Mehmet S. Ismail \\
	\small{Department of Political Economy, King's College London, London, WC2R 2LS, UK} \\ \small{mehmet.s.ismail@gmail.com} \\ \\
	D. Marc Kilgour \\
	\small{Department of Mathematics, Wilfrid Laurier University, Waterloo, Ontario N2L 3C5, Canada} \\ \small{mkilgour@wlu.ca}\\ \\
	Walter Stromquist\\
	\small{132 Bodine Road, Berwyn, PA 19312, USA}\\
	\small{mail@walterstromquist.com}}
\date{}
\maketitle
\newpage
\begin{abstract}
Service sports include two-player contests such as volleyball, badminton, and squash. We analyze four rules, including the Standard Rule (\textit{SR}), in which a player continues to serve until he or she loses. The Catch-Up Rule (\textit{CR}) gives the serve to the player who has lost the previous point---as opposed to the player who won the previous point, as under \textit{SR}. We also consider two Trailing Rules that make the server the player who trails in total score. Surprisingly, compared with \textit{SR}, only \textit{CR} gives the players the same probability of winning a game while increasing its expected length, thereby making it more competitive and exciting to watch. Unlike one of the Trailing Rules, \textit{CR} is strategy-proof. By contrast, the rules of tennis fix who serves and when; its tiebreaker, however, keeps play competitive by being fair---not favoring either the player who serves first or who serves second.
\end{abstract}
\newpage
\section{Introduction}

In service sports, competition between two players (or teams) involves one player serving some object---usually a ball, but a ``shuttlecock" in badminton---which the opponent tries to return.  Service sports include tennis, table tennis (ping pong), racquetball, squash, badminton, and volleyball.

If the server is successful, he or she wins the point; otherwise---in most, but not all, of these sports---the opponent does.  If the competitors are equally skilled, the server generally has a higher probability of winning than the receiver.  We will say more later about what constitutes winning in various service sports.

In some service sports, such as tennis and table tennis, the serving order is fixed---the rules specify when and for how long each player serves. By contrast, the serving order in most service sports, including racquetball, squash, badminton, and volleyball, is variable: It depends on who won the last point.\footnote{In game theory, a game is defined by ``the totality of the rules that describe it" \cite[p. 49]{neumann1944}.  The main difference between fixed-order and variable-order serving rules is that when the serving order is variable, the course of play (order of service) may depend on the results on earlier points, whereas when the order is fixed, then so is the course of play.  Put another way, the serving order is determined exogenously in fixed-order sports but endogenously in variable-order sports.}  In these sports, if the server won the last point, then he or she serves on the next point also, whereas if the receiver won, he or she becomes the new server.  In short, the winner of the last point is the next server. We call this the Standard Rule (\textit{SR}).

In this paper, we analyze three alternatives to \textit{SR}, all of which are variable. The simplest is the
\begin{itemize}
	\item Catch-Up Rule (\textit{CR}): Server is the loser of the previous point---instead of the winner, as under  \textit{SR}.\footnote{The idea of catch-up is incorporated in the game of Catch-Up (see http://game.engineering.nyu.edu/projects/catch-up/ for a playable version), which is analyzed in \cite{isaksen2015}.}
\end{itemize}
\noindent The two other serving rules are Trailing Rules (\textit{TR}s), which we also consider variable: A player who is behind in points becomes the server.  Thus, these rules take into account the entire history of play, not just who won or lost the previous point.  If there is a tie, then who becomes the server depends on the situation immediately prior to the tie:
\begin{itemize}
	\item \textit{TRa}: Server is the player who was ahead in points prior to the tie;
	\item \textit{TRb}: Server is the player who was behind in points prior to the tie.
\end{itemize}

We calculate the players' win probabilities under all four rules. Our only data about the players are their probabilities of winning a point on serve, which we take to be equal exactly when the players are equally skilled. We always assume that all points (rounds) are independent. Among other findings, we prove that \textit{SR}, \textit{CR}, and \textit{TRa} are strategy-proof (or incentive compatible)---neither the server nor the receiver can ever benefit from deliberately losing a point. But \textit{TRb} is strategy-vulnerable: Under \textit{TRb}, it is possible for a player to increase his or her probability of winning a game by losing a point deliberately, under certain conditions that we will spell out.

We analyze the probability that each player wins a game by being the first to score a certain number of points (Win-by-One); later, we analyze Win-by-Two, in which the winner is the first player to score at least the requisite number of points and also to be ahead by at least two points at that time.  We also assess the effects of the different rules on the expected length of a game, measured by the total number of points until some player wins.

Most service sports, whether they use fixed or variable serving rules, use Win-by-Two.  We compare games with and without Win-by-Two to assess the effects of this rule.  Although Win-by-Two may prolong a game, it has no substantial effect on the probability of a player's winning if the number of points needed to win is sufficiently large and the players are equally skilled.  The main effect of Win-by-Two is to increase the drama and tension of a close game.

The three new serving rules give a break to a player who loses a point, or falls behind, in a game.  This change can be expected to make games, especially games between equally skilled players, more competitive---they are more likely to stay close to the end and, hence, to be more exciting to watch.

\textit{CR}, like \textit{SR}, is Markovian in basing the serving order only on the outcome of the previous point, whereas \textit{TRa} and \textit{TRb} take accumulated scores into account.\footnote{\textit{TRa} is the same as the ``behind first, alternating order" mechanism of \cite{Anbarci15}, which says that if a player is behind, it serves next; if the score is tied, the order of service alternates (i.e., switches on the next round).  Serving will alternate under this mechanism only if the player who was ahead prior to a tie---and, therefore, whose opponent served successfully to create the tie---then becomes the new server, causing an alternation in the server.  But this is just \textit{TRa}, as the server is the player who was behind in points prior to the tie.}  The two \textit{TR}s may give an extra advantage to weaker players, which is not true of \textit{CR}. As we will show, under Win-by-One, \textit{CR} gives the players exactly the same probabilities of winning a game as \textit{SR} does.  At the same time, \textit{CR} increases the expected length of games and, therefore, also their competitiveness.  For this reason, our major recommendation is that \textit{CR} replace \textit{SR} to enhance competition in service sports with variable service rules.

\section{Win-by-One}

\subsection{Probability of Winning}
Assume that $A$ has probability $p$ of winning a point when $A$ serves, and $B$ has probability $q$ of winning a point when $B$ serves.  (Warning! In this paper, $q$ is not an abbreviation for $1-p$.)  We always assume that $A$ serves first, and that $0<p<1$ and $0<q<1$.  We will use pronouns ``he'' for $A$, ``she'' for $B$.

We begin with the simple case of Best-of-3, in which the first player to reach 2 points wins.  There are three ways for $A$ to win: (1) by winning the first two points (which we denote $\winsequence{AA}$), (2) by winning the first and third points ($\winsequence{ABA}$), and (3) by winning the second and third points ($\winsequence{BAA}$).

Any of these ``win sequences'' can occur under any of the four rules described in the Introduction (\textit{SR}, \textit{CR}, \textit{TRa}, \textit{TRb}), but each rule would imply a different sequence of servers.  A win sequence tells us which player won each point, but it does not tell us which player served each point.  For this purpose we define an ``outcome'' as a win sequence that includes additional information: We use $A$ or $B$ when either player wins when serving, and $\overmark{A}$ and $\overmark{B}$ when either player wins when the other player is serving.  Thus, a bar over a letter marks a server loss.

For example, under \textit{SR}, the win sequence $\winsequence{AA}$ occurs when $A$ wins two serves in a row, and so the outcome is $AA$ (no server losses).  The win sequence $\winsequence{ABA}$ occurs when $A$ serves the first point and wins, $A$ retains the serve for the second point and loses, and then $B$ gains the serve for the third point and loses, giving the outcome $A\overmark{B}\overmark{A}$ (two server losses).  Under \textit{SR} the win sequence $\winsequence{BAA}$ corresponds to the outcome $\overmark{B}\overmark{A}A$.

Because we assume that serves are independent events---in particular, not dependent on the score at any point---we can calculate the probability of an outcome by multiplying the appropriate probabilities serve by serve.  For example, the outcome $AA$ occurs with probability $p^2$, because it requires two server wins by $A$.  The outcome $A\overmark{B}\overmark{A}$ has probability $p(1 - p)(1 - q)$, and the outcome $\overmark{B}\overmark{A}A$ has probability $(1-p)(1-q)p$.  Adding these three probabilities gives the total probability that $A$ wins Best-of-3 under \textit{SR}:
\[ Pr_{SR}(A) = p^2 + p(1-p)(1-q) + (1-p)(1-q)p = 2p-p^2-2pq+2p^2q. \] 

The translation from outcome to probability is direct: $A$, $B$, $\overmark{A}$, $\overmark{B}$ correspond to probabilities $p$, $q$, $(1-q)$, $(1-p)$, respectively, whatever the rule.  Under \textit{CR}, the win sequences $\winsequence{AA}$, $\winsequence{ABA}$, $\winsequence{BAA}$ correspond to the outcomes $A\overmark{A}$ (probability $p(1-q)$), $ABA$ (probability $pqp$) and $\overmark{B}A\overmark{A}$ (probability $(1-p)p(1-q)$) respectively. So the probability that $A$ wins Best-of-3 under \textit{CR} is
\[ Pr_{CR}(A) = p(1-q) + pqp + (1-p)p(1-q) = 2p-p^2-2pq+2p^2q. \]
Thus, the totals for \textit{SR} and \textit{CR} are the same, even though the probabilities being added are different.

Table~\ref{table1} extends this calculation to \textit{TRa} and \textit{TRb}.  It shows that $A$ has the same probability of winning when using \textit{SR}, \textit{CR}, and \textit{TRa}, but a different probability when using \textit{TRb}.  The latter probability is generally smaller:
\[ Pr_{SR}(A) = Pr_{CR}(A) = Pr_{TRa}(A)  \geq Pr_{TRb}(A),\]
with equality only when $p+q=1$. The intuition behind this result is that \textit{TRb} most helps the player who falls behind.  This player is likely to be $B$ when $A$ serves first, making $B$'s probability of winning greater, and $A$'s less, than under the other rules. (Realistically, serving is more advantageous than receiving in most service sports, though volleyball is an exception, as we discuss later.)

{
	\renewcommand{\arraystretch}{1.2}
	\begin{table}[]
		\begin{center}
			\scalebox{0.85}{
				\begin{tabular}{|l|c|c|c|c|}
					\hline
					&        \multicolumn{3}{c|}{Probability that $A$ wins}   & Expected Length \\ \hline \hline
					\text{Win sequence:}   & $\winsequence{AA}$ & $\winsequence{ABA}$ & $\winsequence{BAA}$ &  \\ \hline \hline
					\textbf{\emph{SR}} & \multicolumn{3}{l|}{}  &  \\ \hline
					Outcome     & $AA$ & $A\overmark{B}\overmark{A}$  &  $\overmark{B}\overmark{A}A$ &     \\ \hline
					Probability & $p^2$ & $p(1-p)(1-q)$ & $(1-p)(1-q)p$  & \\ \hline
					Sum         & \multicolumn{3}{c|}{$2p-p^2-2pq+2p^2q$}  &  $3-q-p^2+pq$\\ \hline\hline
					\textbf{\emph{CR}} & \multicolumn{3}{l|}{}  &  \\ \hline
					Outcome     & $A\overmark{A}$ & $ABA$  &  $\overmark{B}A\overmark{A}$   &   \\ \hline
					Probability &$p(1-q)$ & $pqp$ & $(1-p)p(1-q)$  &  \\ \hline
					Sum         & \multicolumn{3}{c|}{$2p-p^2-2pq+2p^2q$}  & $2+p-p^2+pq$ \\ \hline\hline
					\textbf{\emph{TRa}} & \multicolumn{3}{l|}{}  &  \\ \hline
					Outcome     & $A\overmark{A}$ & $ABA$  &  $\overmark{B}A\overmark{A}$  &  \\ \hline
					Probability &$p(1-q)$ & $pqp$ & $(1-p)p(1-q)$ &   \\ \hline
					Sum         & \multicolumn{3}{c|}{$2p-p^2-2pq+2p^2q$} &  $2+p-p^2+pq$ \\ \hline\hline
					\textbf{\emph{TRb}} & \multicolumn{3}{l|}{}  &  \\ \hline
					Outcome     & $A\overmark{A}$ & $AB\overmark{A}$  &  $\overmark{B}AA$ &    \\ \hline
					Probability & $p(1-q)$ & $pq(1-q)$ & $(1-p)p^2$ & \\ \hline
					Sum         & \multicolumn{3}{c|}{$p+p^2-p^3-pq^2$} &  $2+p-p^2+pq$ \\ \hline
			\end{tabular} }
			\vspace{2mm}
			\caption{Probability that $A$ wins and Expected Length for a Best-of-3 game}
			\label{table1}
		\end{center}
	\end{table}
	
	But what if a game goes beyond Best-of-3?  Most service sports require that the winner be the first player to score 11, 15, or 21 points, not 2.  Even for Best-of-5, wherein the winner is the first player to score 3 points, the calculations are considerably more complex and tedious.  Instead of three possible ways in which each player can win, there are ten. We carried out these calculations and found that
	\[ Pr_{SR} (A)  =  Pr_{CR} (A) \geq Pr_{TRa} (A),\]
	with equality only in the case of $p+q=1$.  When the players are of equal strength ($p=q$) we also have $Pr_{TRa} (A) \geq Pr_{TRb}(A),$ with equality only when $p = 1/2$.  But in the Best-of-5 case, when $p \neq q$, no inequality holds generally.  We can even have $Pr_{TRb}(A) > Pr_{SR}(A)$ for certain values of $p$ and $q$.
	
	But does the equality of $A$'s winning under \textit{SR} and \textit{CR} hold generally?  Theorem 1 below shows that this is indeed the case.
	
	\vspace{3mm}
	\noindent \textbf{Theorem 1.}  \emph{Let $k \geq 1$.  In a Best-of-$(2k+1)$ game, $Pr_{SR}(A) = Pr_{CR}(A)$.}
	\vspace{2mm}
	
	For a proof, see the Appendix. Kingston \cite{kingston1976} proved that, in a Best-of-($2k+1$) game with $k \geq 1$, the probability of $A$'s winning under \textit{SR} is equal to the probability of $A$'s winning under any fixed rule that assigns $k+1$ serves to  $A$, and $k$ serves to $B$. This proof was simplified and made more intuitive by Anderson \cite{anderson1977}; it appears in slightly different form in a recent book \cite[Problem 90]{hess2014} and a forthcoming paper \cite[Problem 4]{winkler2018}. Independently, we found the same result and extended it to the case in which a player's probability of winning a point when he or she serves is variable.%
	\footnote{We have assumed that $A$ wins each of his serves with the same probability $p$, and that $B$ wins each of her serves with the same probability $q$. But what if $A$ wins his $i$th serve with probability $p_i$, for $i=1,2,\ldots$, and that $B$ wins her $j$th serve with probability $q_j$, for $j=1,2,\ldots$? Then the proof of Theorem 1 still applies. We can even make $A$'s probability $p_i$ depend on the results of $A$'s previous serves, and similarly for $B$. As an example, this more general model seems reasonable for volleyball, in which a team changes individual servers each time the team loses a serve.}
	
	The basis of the proof is the idea of serving schedule, which is a record of server wins and server losses organized according to the server. The schedule lists the results of $k+1$ serves by $A$, and then $k$ serves by $B$. To illustrate, the description of a Best-of-3 game requires a server schedule of length 3, consisting of a record of whether the server won or lost on $A$'s first and second serves, and on $B$'s first serve. The serving schedule $(W, L, L)$, for instance, records that $A_1 = W$ (i.e., $A$ won on his first serve), $A_2 = L$ ($A$ lost on his second serve), and $B_1 = L$ ($B$ lost on her first serve).  The idea is that, if the serving schedule is fixed, then both serving rules, \textit{SR} and \textit{CR}, give the same outcome as an Auxiliary Rule (\textit{AR}), in which $A$ serves twice and then $B$ serves once. Specifically,
	
	\begin{itemize}
		\item	\textit{AR}: $(A_1 = W, A_2 = L, B_1 = L )$, outcome $A\overline{BA}$. $A$ wins 2-1.
		\item	\textit{SR}: $(A_1 = W, A_2 = L, B_1 = L )$, outcome $A\overline{BA}$. $A$ wins 2-1.
		\item	\textit{CR}: $(A_1 = W, B_1 = L )$, outcome $A\overline{A}$. $A$ wins 2-0.
	\end{itemize}
	
	\noindent Observe that the winner is the same under each rule, despite the differences in outcomes and scores. The basis of our proof is a demonstration that, if the serving schedule is fixed, then the winners under \textit{AR}, \textit{SR}, and \textit{CR} are identical.
	
	The \textit{AR} service rule is particularly simple and permits us to establish a formula to determine win probabilities.  The fact that the \textit{AR}, \textit{SR}, and \textit{CR} service rules have equal win probabilities makes this representation more useful.
	
	\begin{cor}
		The probability that $A$ wins a Best-of-$(2k+1)$ game under any of the service rules \textit{SR}, \textit{CR}, or \textit{AR} is
		\begin{align}
		Pr_{AR}(A)&=Pr_{SR}(A)=Pr_{CR}(A) \notag\\
		&=\sum_{n=1}^{k+1}\sum_{m=0}^{n-1} p^n (1-p)^{k+1-n} q^m (1-q)^{k-m}  \binom{k+1}{n} \binom{k}{m}.\notag
		\end{align}
	\end{cor}
	\begin{proof}
		As shown in the proof of Theorem 1, each of these probabilities is equal to the probability of choosing a service schedule in which $A$ has at least $k+1$ total wins.  Suppose that $A$ has exactly $n$ server wins among his $k{+}1$ serves.  If $n = 0$, $A$ must lose under any of the rules.  If $n = 1, 2, \ldots, k+1$, then $A$ wins if and only if $B$ has $m$ server wins among her first $k$ serves, with $m \leq n-1$.  The probability above follows directly.
	\end{proof}
	
	\subsection{Expected Length of a Game}
	The different service rules may affect not only the probability of $A$'s winning but also the expected length (\textit{EL}) of a game.  To illustrate the latter calculation, consider \textit{SR} for Best-of-3 in the case $p=q$.  Clearly, the game lasts either two or three serves. Table 1 shows that $A$ will win with probability $p^2$ after 2 serves; also, $B$ will win with probability $p(1-p)$ after 2 serves. Therefore, the probability that the game ends after two serves is $p^2+p(1-p) = p$, so the probability that it ends after three serves is $(1-p)$.  Hence, the expected length is
	\[ EL_{SR} = 2p+3(1 - p) = 3-p. \]
	
	To illustrate, if $p=0$, servers always lose, so the game will take 3 serves---and 2 switches of server---before the player who starts ($A$) loses 2 points to 1.  On the other hand, if $p = 1$, $EL_{SR}$ = 2, because $A$ will win on his first two serves.
	
	By a similar calculation, $EL_{CR} = 2+p$. We also give results for the two \textit{TR}s in Table 1, producing the following ranking of expected lengths for Best-of-3 games if $p > \frac12$,
	\[ EL_{TRb} = EL_{TRa} = EL_{CR} > EL_{SR}. \]
	
	\noindent Observe that the expected length of a game for Best-of-3 is a minimum under \textit{SR}. To give an intuition for this conclusion, if $A$ is successful on his first serve, he can end play with a second successful serve, which is fairly likely if $p$ is large. On the other hand, \textit{CR} and the \textit{TR}s shift the service to the other player, who now has a good chance of evening the score if $p > \frac12$.
	
	Calculations for Best-of-5 games show that
	\[ EL_{TRb} = EL_{TRa}  > EL_{CR} > EL_{SR}.\]
	
	\noindent Thus, both \textit{TR}s have a greater expected length than \textit{CR}, and the length of \textit{CR} in turn exceeds the length of \textit{SR}.\footnote{In an experiment, Ruffle and Volij \cite{ruffle2015} found that the average length of a Best-of-9 game was greater under \textit{CR} than under \textit{SR}.}
	
	\vspace{3mm}
	\noindent\textbf{Theorem 2.}  \emph{In a Best-of-$(2k+1)$ game for any $k \geq 1$ and $0 < p < 1$, $0 < q < 1$, the expected length of a game is greater under \textit{CR} than under \textit{SR} if and only if $p+q > 1$.}
	\vspace{3mm}
	
	\noindent The proof of Theorem 2 appears in the Appendix. The proof involves a long string of sums that we were unable to simplify or evaluate. However, we were able to manipulate them enough to prove the theorem. A more insightful proof would be welcome.
	
	Henceforth, we focus on the comparison between \textit{SR} and \textit{CR} for two reasons:
	\begin{enumerate}
		\item Under \textit{CR}, the probability of a player's winning is the same as under \textit{SR} (Theorem 1).
		\item Under \textit{CR}, the length of the game is greater (in expectation) than under \textit{SR}, provided $p+q > 1$ (Theorem 2).
	\end{enumerate}
	\noindent For these reasons, we believe that most service sports currently using the \textit{SR} rule would benefit from the \textit{CR} rule.  Changes should not introduce radical shifts, such as changing the probability of winning.  At the same time, \textit{CR} would make play appear to be more competitive and, therefore, more likely to stimulate fan (and player) interest.  \textit{CR} satisfies both of our criteria.
	
	\textit{CR} keeps games close by giving a player who loses a point the opportunity to serve and, therefore, to catch up, given $p + q > 1$.  Consequently, the expected length of games will, on average, be greater under \textit{CR} than \textit{SR}.  For Best-of-3 games, if $p = q = \frac23$, then $EL_{CR} = \frac83$, whereas $EL_{SR} = \frac73$.  Still, the probability that $A$ wins is the same under both rules ($\frac{16}{27} = 0.592$, from Table 1).  By Theorem 1, each player can rest assured that \textit{CR}, compared with \textit{SR}, does not affect his or her chances of winning.
	
	It is true under \textit{SR} and \textit{CR} that if $p = \frac23$, then $A$ has almost a 3:2 advantage in probability of winning ($\frac{16}{27} = 0.592$) in Best-of-3.  But this is less than the 2:1 advantage $A$ would enjoy if a game were decided by just one serve.  Furthermore, $A$'s advantage drops as $k$ increases for Best-of-($2k+1$), so games that require more points to win tend to level the playing field. In Best-of-5, $A$'s winning probability drops to $0.568$.
	
	Because the two \textit{TR}s do even better than \textit{CR} in lengthening games, would not one be preferable in making games closer and more competitive?  The answer is ``yes," but they would reduce $Pr(A)$, relative to \textit{SR} (and \textit{CR}), and so would be a more significant departure from the present rule.  For Best-of-3, $Pr_{TRb} (A) = \frac{14}{27} = 0.519$, which is 12.5\% lower than $Pr_{SR}(A) = Pr_{CR}(A) = \frac{16}{27} = 0.592$.\footnote{We recognize that it may be desirable to eliminate entirely the first-server advantage, but this is difficult to accomplish in a service sport, as the ball must be put into play somehow.  In fact, the first-server advantage decreases as the number of points required to win increases. Of course, any game is ex-ante fair if the first server is chosen randomly according to a coin toss, but it would be desirable to align ex-ante and ex-post fairness.  As we show later, the serving rules of the tiebreaker in tennis achieve this alignment exactly.}  But \textit{TRb} has a major strike against it that the other rules do not, which we explore next.
	
	\subsection{Incentive Compatibility}
	As discussed in the Introduction, a rule is \textit{strategy-proof} or \textit{incentive compatible} if no player can ever benefit from deliberately losing a point; otherwise, it is \textit{strategy-vulnerable}.\footnote{For an informative analysis of strategizing in sports competitions, and a discussion of its possible occurrence in the 2012 Olympic badminton competition, see \cite{pauly2014}.}
	
	\vspace{3mm}
	\noindent\textbf{Theorem 3.} \emph{\textit{TRb} is strategy-vulnerable, whereas \textit{SR} and \textit{CR} are strategy-proof. \textit{TRa} is strategy-proof whenever $p+q>1$.}
	\vspace{3mm}
	
	\begin{proof}
		To show that \textit{TRb} is strategy-vulnerable, assume a Best-of-3 game played under \textit{TRb}. We will show that there are values of $p$ and $q$ such that $A$ can increase his probability of winning the game by deliberately losing the first point.
		
		Recall that $A$ serves first. If $A$ loses on the initial serve, the score is 0-1, so under \textit{TRb}, $A$ serves again.  If he wins the second point (with probability $p$), he ties the score at 1-1 and  serves once more, again winning with probability $p$. Thus, by losing deliberately, $A$ wins the game with probability $p^2$.
		
		If $A$ does not deliberately lose his first serve, then there are three outcomes in which he will win the game:
		\begin{itemize}
			\item $A\overmark{A}$ with probability $p(1-q)$
			\item $AB\overmark{A}$ with probability $pq(1-q)$
			\item $\overmark{B}AA$ with probability $(1-p)p^2$.
		\end{itemize}
		\noindent Therefore, $A$'s probability of winning the game is greater when $A$ deliberately loses the first serve if and only if
		\[ p^2 > (p-pq)+(pq-pq^2)+(p^2-p^3),\]
		
		\noindent which is equivalent to $p(1-p^2-q^2) < 0$.  Because $p > 0$, it follows that $A$ maximizes the probability that he wins the game by deliberately losing his first serve when
		\[p^2+q^2 > 1. \]
		
		In $(p, q)$-space, this inequality describes the exterior of a circle of radius 1 centered at $(p, q) = (0, 0)$. Because the probabilities $p$ and $q$ can be any numbers within the unit square of this space, it is possible that $(p, q)$ lies outside this circle. If so, $A$'s best strategy is to deliberately lose his initial serve.
		
		We next show that \textit{SR} is strategy-proof. In a Best-of-$(2k+1)$ game played under \textit{SR}, let $(C, x, y)$ denote a state in which player $C$ ($C =A \text{ or } B$) is the server, $A$'s score is $x$, and $B$'s is $y$. For example, when the game starts, the state is $(A, 0, 0)$. Let $W_{AS}(C, x, y)$ denote $A$'s win probability from state $(C, x, y)$ under \textit{SR}. It is clear that $A$'s win probability $W_{AS}(A, x, y)$ is increasing in $x$.
		
		To show that a game played under \textit{SR} is strategy-proof for $A$, we must show that $W_{AS}(A, x + 1, y) \geq W_{AS}(B, x, y + 1)$ for any $x$ and $y$. Now
		\[W_{AS}(A, x, y) = pW_{AS}(A, x + 1, y) + (1-p)W_{AS}(B, x, y + 1),\]
		which implies that
		\begin{multline*}
		p[W_{AS}(A, x+1, y) - W_{AS}(A, x, y)] \\
		+ (1 - p)[W_{AS}(B, x, y + 1) - W_{AS}(A, x, y)] = 0.
		\end{multline*}
		Because $0 < p < 1$ and $W_{AS}(A, x + 1, y) \geq W_{AS}(A, x, y)$, the first term is nonnegative, so the second must be nonpositive. It follows that
		\[ W_{AS}(B, x, y + 1) \leq W_{AS}(A, x, y) \leq W_{AS}(A, x+1, y), \]
		as required. Thus, $A$ cannot gain under \textit{SR} by deliberately losing a serve, so \textit{SR} is strategy-proof for $A$. By an analogous argument, \textit{SR} is also strategy-proof for $B$.
		
		We next show that \textit{CR} is strategy-proof. We use the same notation for states, and denote $A$'s win probability from state $(C, x, y)$ by $W_{AC}(C, x, y)$ and $B$'s by $W_{BC}(C, x, y)$. As in the \textit{SR} case, it is clear that  $W_{AC}(C, x, y)$ is an increasing function of $x$.
		
		To show that a game played under \textit{CR} is strategy-proof for $A$, we must show that $W_{AC}(B, x+1, y) \geq W_{AC}(A, x, y + 1)$ for any $x$ and y. Now
		\[W_{AC}(A, x, y) = pW_{AC}(B, x+1, y) + (1-p)W_{AC}(A, x, y+1),\]
		which implies that
		\begin{multline*}
		p\left[W_{AC}(B, x+1, y) - W_{AC}(A, x, y)\right] \\
		+ (1 - p)\left[W_{AC}(A, x, y + 1) - W_{AC}(A, x, y)\right] = 0.
		\end{multline*}
		Again, it follows that
		\[ W_{AC}(A, x, y + 1) \leq W_{AC}(A, x, y) \leq W_{AC}(B, x+1, y), \]
		as required. Thus a game played under \textit{CR} is strategy-proof for $A$, and by analogy it is strategy-proof for $B$.
		
		The proof that \textit{TRa} is strategy-proof is left for the Appendix.
	\end{proof}
	
	To illustrate the difference between \textit{TRa} and \textit{TRb} and indicate its implications for strategy-proofness, suppose that the score is tied and that $A$ has the serve. If $A$ deliberately loses, $B$ will be ahead by one point, so $A$ will serve the second point. If $A$ is successful, the score is again tied, and \textit{TRb} awards the next serve to $A$, whereas \textit{TRa} gives it to $B$, who was ahead prior to the most recent tie. If $p > 1 - q$, the extra serve is an advantage to $A$.
	
	Suppose that the players are equally skilled at serving ($p=q$) in the Best-of-3 counterexample establishing that \textit{TRb} is not strategy-proof.  Then the condition for deliberately losing to be advantageous under \textit{TRb} is $p > 0.707$.  Thus, strategy-vulnerability arises in this simple game if both players have a server win probability greater than about 0.71, which is high but not unrealistic in most service sports.
	
	From numerical calculations, we know that for Best-of-5 (and longer) games, neither \textit{TRa} (nor the strategy-vulnerable \textit{TRb}) has the same win probability for $A$ as \textit{SR}, whereas by Theorem 1 \textit{CR} maintains it in Win-by-One games of any length.  This seems a good reason for focusing on \textit{CR} as the most viable alternative to \textit{SR}, especially because our calculations show that \textit{CR} increases the expected length of Best-of-5 and longer games, thereby making them appear more competitive.
	
	Most service sports are not Win-by-One but Win-by-Two (racquetball is an exception, discussed in Section 3).  If $k+1$ points are required to win, and if the two players tie at $k$ points, then one player must outscore his or her opponent by two points in a tiebreaker in order to win.
	
	We next analyze Win-by-Two's effect on $A$'s probability of winning and the expected length of a game.\footnote{Before tiebreakers for sets were introduced into tennis in the 1970s, Kemeny and Snell \cite[pp. 161--164]{kemeny1983} showed how the effect of being more skilled in winning a point in tennis ramifies to a game, a set, and a match.  For example, a player with a probability of 0.51 (0.60) of winning a point---whether the player served or not---had a probability of 0.525 (0.736) of winning a game, a probability of 0.573 (0.966) of winning a set, and a probability of 0.635 (0.9996) of winning a match (in men's competition, or Best-of-5 for sets).  In tennis, to win a game or a set now requires, respectively, winning by at least two points or at least two games (if there is no tiebreaker).  In effect, tiebreakers change the margin by which a player must win a set from at least two games to at least two points in the tiebreaker.}  We note that the service rule, which we take to be \textit{SR} or \textit{CR}, affects the tiebreak; starting in a tied position, \textit{SR} gives one player the chance of winning on two consecutive serves, whereas under \textit{CR} a player who loses a tiebreak must lose at least once on serve.  We later consider sports with fixed rules for serving and ask how Win-by-Two affects them.
	
	\section{Win-by-Two}
	
	To illustrate Win-by-Two, consider Best-of-3, in which a player wins by being the first to receive 2 points \emph{and} by achieving at least two more points than the opponent.  In other words, 2-0 is winning but 2-1 is not; if a 1-1 tie occurs, the winner will be the first player to lead by 2 points and will thus require more than 3 points.
	
	\subsection{Standard Rule Tiebreaker}
	Consider the ways in which a 1-1 tie can occur under \textit{SR}.  There are two sequences that produce such a tie:
	
	\begin{itemize}
		\item $A\overmark{B}$ with probability $p(1-p)$
		\item $\overmark{B}\overmark{A}$ with probability $(1-p)(1-q)$.
	\end{itemize}
	
	\noindent Thus, the probability of a 1-1 tie is
	\[ (p - p^2) + (1 - p - q + pq) = 1 + pq - p^2 - q, \]
	or $1-p$ if $p=q$.  As $p$ approaches 1 (and $p=q$), the probability of a tie approaches 0, whereas if $p = \frac23$, the probability is $\frac13$.
	
	We now assume that $p=q$ (as we will do for the rest of the paper).  In a Best-of-$(2k+1)$ tiebreaker played under \textit{SR}, recall that $Pr_{SR}(A)$ is the probability that $A$ wins when he serves first.  Then $1-Pr_{SR}(A)$ is the probability that $B$ wins when $A$ serves first.  Since $p=q$,  $1-Pr_{SR}(A)$ is also the probability that $A$ wins the tiebreaker when $B$ serves first. Then it follows that $Pr_{SR}(A)$ must satisfy the following recursion:
	\[ Pr_{SR}(A) =  p^2 + [p(1-p)(1-Pr_{SR}(A))] + [(1-p)^2Pr_{SR}(A)]. \]
	The first term on the right-hand side gives the probability that $A$ wins the first two points.  The second term gives the probability of sequence $A\overmark{B}$---so $A$ wins initially and then loses, recreating a tie---times the probability that $A$ wins when $B$ serves first in the next tiebreaker.  The third term gives the probability of sequence $\overmark{B}\overmark{A}$---in which $A$ loses initially and then $B$ loses, recreating a tie---times the probability that $A$ wins the next tiebreaker. Because $p > 0$, this equation can be solved for $Pr_{SR}(A)$ to yield
	\[ Pr_{SR}(A) = \frac{ 1}{3-2p}. \]
	
	Table 2 presents the probabilities and expected lengths under \textit{SR} and \textit{CR} corresponding to various values of $p$. Under \textit{SR}, observe that $A$'s probability of winning is greater than $p$ when $p = \frac13$ and less when $p = \frac23$.  More generally, it is easy to show that $Pr_{SR}(A) > p$ if $0 < p < \frac12$, and $Pr_{SR}(A)  < p$ if $\frac12 < p < 1$.  Thus, $A$ does worse in a tiebreaker than if a single serve decides a tied game when $\frac12 < p < 1$.
	When $p = 1$, $A$ will win the tiebreaker with certainty, because he serves first and will win every point.  At the other extreme, when $p$ approaches 0, $Pr_{SR}(A)$ approaches $\frac13$. 
	
	When there is a tie just before the end of a game, Win-by-Two not only prolongs the game over Win-by-One but also changes the players' win probabilities.  If $p = \frac23$, under Win-by-One, sequence $A\overmark{B}$ has twice the probability of producing a tie as $\overmark{B}\overmark{A}$. Consequently, $B$ has twice the probability of serving the tie-breaking point (and, therefore, winning) than $A$ does under Win-by-One.
	
	But under Win-by-Two, if $p = \frac23$ and $B$ serves first in the tiebreaker, $B$'s probability of winning is not $\frac23$, as it is under Win-by-One, but $\frac35$. Because we assume $p > \frac12$, the player who serves first in the tiebreaker is hurt---compared with Win-by-One---except when $p = 1$.\footnote{In gambling, the player with a higher probability of winning individual games does better the more games are played.  Here, however, when the players have the same probability $p$ of winning points when they serve, the player who goes first, and therefore would seem to be advantaged under \textit{SR} when $p > \frac12$, does not do as well under Win-by-Two as under Win-by-One.  This apparent paradox is explained by the fact that Win-by-Two gives the second player a chance to come back after losing a point, an opportunity not available under Win-by-One.}
	\begin{table}
		\center
		\begin{tabular}{|c||c|c|c|c|}
			\hline
			$p$ & $Pr_{SR}(A)$ & $Pr_{CR}(A)$ & $EL_{SR}$ & $EL_{CR}$ \\
			\hline \hline
			0 & undefined & 0.00 & $\infty$ & 2.00  \\
			\hline
			1/4 &0.40 & 0.33 &8.00 & 2.67  \\
			\hline
			1/3 &0.43 &0.40 &6.00 &3.00  \\
			\hline
			1/2 & 0.50 &0.50 &4.00 &4.00 \\
			\hline
			2/3 & 0.60 &0.57 &3.00 &6.00\\
			\hline
			3/4 & 0.67 &0.60 &2.67 & 8.00
			\\
			\hline
			1 & 1.00 &undefined &2.00 &$\infty$  \\
			\hline
		\end{tabular}
		\vspace{2mm}
		\caption{Probability ($Pr$) that $A$ wins, and Expected Length of a Game (\textit{EL}), under \textit{SR} and \textit{CR} for various values of $p$.}
	\end{table}
	By how much, on average, does Win-by-Two prolong a game?  Recall our assumption that $p=q$.  It follows that the expected length (\textit{EL}) of a tiebreaker does not depend on which player serves first (for notational convenience below, we assume that $A$ serves first).  The following recursion gives the expected length (\textit{EL}) under \textit{SR} when $p > 0$:
	\[ EL_{SR} =  [p^2+(1-p)p](2) + [p(1-p)+(1-p)^2](EL_{SR}+2).\]
	The first term on the right-hand side of the equation contains the probability that the tiebreaker is over in two serves (in sequences $AA$ and $\overmark{B}B$). The second term contains the probability that the score is again tied after two serves (in sequences $A\overmark{B}$ and $\overmark{B}\overmark{A}$), which implies that the conditional expected length is $EL_{SR}+2$.
	
	Assuming that $p > 0$, this relation can be solved to yield
	\[ EL_{SR}(A) = \frac{ 2}{p}. \]
	Clearly, if $p = 1$, $A$ immediately wins the tiebreaker with 2 successful serves, whereas the length increases without bound as $p$ approaches 0; in the limit, the tiebreaker under \textit{SR} never ends.
	
	\subsection{Catch-Up Rule Tiebreaker}
	We suggested earlier that \textit{CR} is a viable alternative to \textit{SR}, so we next compute the probability that $A$ wins under \textit{CR} using the following recursion, which mirrors the one for \textit{SR}:
	\[Pr_{CR}(A) =  p(1-p) + p^2 Pr_{CR}(A) + [(1-p)(p)[1-Pr_{CR} (A)]. \]
	
	\noindent The first term on the right-hand side is the probability of sequence $A\overmark{A}$, in which $A$ wins the first point and $B$ loses the second point, so $A$ wins at the outset.  The second term gives the probability of sequence $AB$---$A$ wins the first point and $B$ the second, creating a tie---times the probability that $A$ wins eventually.  The third term gives the probability of sequence $\overmark{B}A$---so $A$ loses initially and then wins, creating a tie---times the probability that $A$ wins when $B$ serves first in the tiebreaker, the complement of the probability that $A$ wins when he serves first.
	
	This equation can be solved for $Pr_{CR}(A)$ provided $p < 1$, in which case
	\[ Pr_{CR}(A)=\frac{2p}{1+2p}.\]
	Note that when $p = 1$, $Pr_{CR}(A)$ is undefined---not $\frac23$---because the game will never end since each win by one player leads to a win by the other player, precluding either player from ever winning by two points. For various values of $p < 1$, Table 2 gives the corresponding probabilities.
	
	What is the expected length of a tiebreaker under \textit{CR}?  When $p=q < 1$, $EL_{CR}$ satisfies the recursion
	\[ EL_{CR} =  [p(1-p)+(1-p)^2](2) + [p^2+(1-p)p](EL_{CR}+2), \]
	which is justified by reasoning similar to that given earlier for $EL_{SR}$.  This equation can be solved for $EL_{CR}$ only if $p < 1$, in which case
	\[ EL_{CR}=\frac{2}{1-p}. \]
	Unlike $EL_{SR}$, $EL_{CR}$ increases with $p$, but it approaches infinity at $p = 1$, because the tiebreaker never ends when both players alternate successful serves.
	
	\subsection{Comparison of Tiebreakers}
	We compare $Pr_{SR}(A)$ and $Pr_{CR}(A)$ in Figure 1.  Observe that both are increasing in $p$, though at different rates.  The two curves touch at $p = \frac12$, where both probabilities equal $\frac 12$.  But as $p$ approaches 1, $Pr_{SR}(A)$ increases at an increasing rate toward 1, whereas $Pr_{CR}(A)$ increases at a decreasing rate toward $\frac23$.  When $p > \frac12$, $A$ has a greater advantage under \textit{SR}, because he is fairly likely to win at the outset with two successful serves, compared to under \textit{CR} where, if his first serve is successful, $B$ (who has equal probability of success) serves next.
	
	\begin{figure}
		\begin{center}
			\includegraphics[width=0.88\textwidth]{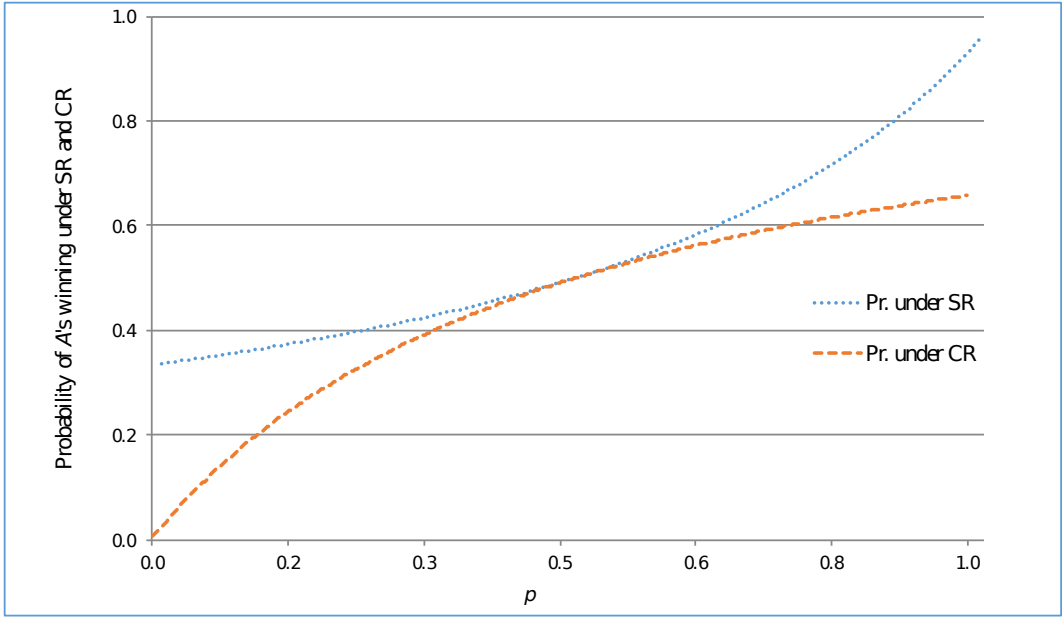}
			\caption{Graph of $Pr_{SR}(A)$ and $Pr_{CR}(A)$ as a function of $p$.}
			\label{fig1}
		\end{center}
	\end{figure}
	
	The graph in Figure 2 shows the inverse relationship between $EL_{SR}$ and $EL_{CR}$ as a function of $p$. If $p = \frac23$, which is a realistic value in several service sports, the expected length of the tiebreaker is greater under \textit{CR} than under \textit{SR} (6 vs. 3), which is consistent with our earlier finding for Win-by-One: $EL_{CR} > EL_{SR}$ for Best-of-$(2k+1)$ (Theorem 2).  Compared with \textit{SR}, \textit{CR} adds 3 serves, on average, to the tiebreaker, and also increases the expected length of the game prior to any tiebreaker, making a tiebreaker that much more likely.
	
	\begin{figure}
		\begin{center}
			\includegraphics[width=0.88\textwidth]{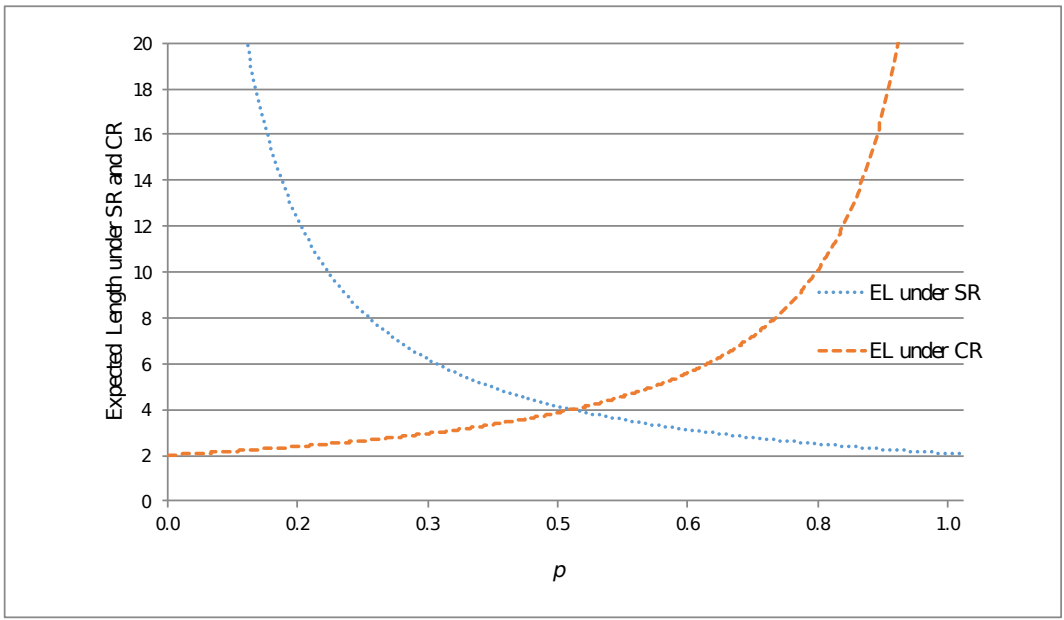}
			\caption{Graph of $EL_{SR}$ and $EL_{CR}$ as a function of $p$.}
			\label{fig2}
		\end{center}
	\end{figure}
	
	But we emphasize that in the Win-by-One Best-of-($2k+1$) game, $Pr_{SR}(A) = Pr_{CR}(A)$ (Theorem 1).  Thus, compared with \textit{SR}, \textit{CR} does not change the probability that $A$ or $B$ wins in the regular game.  But if there is a tie in this game, the players' win probabilities may be different in an \textit{SR} tiebreaker versus a \textit{CR} tiebreaker. For example, if $p = \frac23$ and $A$ is the first player to serve in the tiebreaker, he has a probability of $\frac35 = 0.600$ of winning under \textit{SR} and a probability of $\frac47 = 0.571$ of winning under \textit{CR}.
	
	\subsection{Numerical Comparisons}
	Clearly, serving first in the tiebreaker is a benefit under both rules.  For Best-of-3, we showed in Section 2 that if $p = \frac23$, the probability of a tie is $\frac13$, and $B$ is twice as likely as $A$ to serve first in the tiebreaker under \textit{SR}.
	
	In Table 3, we extend Best-of-3 to the more realistic cases of Best-of-11 and Best-of-21 for both Win-by-One (WB1) and Win-by-Two (WB2) for the cases $p = \frac23$ and $p = \frac34$.  Under WB1, the first player to score $m$ = 6 or 11 points wins; under WB2, if there is a 5-5 or 10-10 tie, there is a tiebreaker, which continues until one player is ahead by 2 points.
	
	We compare $Pr_{SR}(A)$ and $Pr_{CR}(A)$, which assume WB1, with the probabilities $Qr_{CR}(A)$ and $Qr_{CR}(A)$, which assume WB2, for $p = \frac23$ and $\frac34$. The latter probabilities take into account the probability of a tie, $Pr(T)$, which adjusts the probabilities of $A$ and $B$'s winning.
	
	For the values of $m = 3$, $11$, and $21$ and $p = \frac23$ and $\frac34$, we summarize below how the probabilities of winning and the expected lengths of a game are affected by \textit{SR} and \textit{CR} and Win-by-One and Win-by-Two:
	
	\begin{enumerate}
		\item As $m$ increases from $m = 3$ to 11 to 21, the probability of a tie, $Pr(T)$, decreases by more than a factor of two for both $p = \frac23$ and $p = \frac34$.  Even at $m = 21$, however, $Pr(T)$ is always at least 12\% for \textit{SR} and at least 25\% for \textit{CR}, indicating that, especially for \textit{CR}, Win-by-Two will often end in a tiebreaker.
		
		\item For $p = \frac23$, the probability of a tie, $Pr(T)$, is twice as great under \textit{CR} as \textit{SR}; it is three times greater for $p = \frac34$.  Thus, games are much closer under \textit{CR} than \textit{SR} because of the much greater frequency of ties under \textit{CR}.
		
		\item For Best-of-11 and Best-of-21, the probability of $A$'s winning under Win-by-One ($Pr(A)$) and Win-by-Two ($Qr(A)$) falls within a narrow range (53--57\%), whether \textit{CR} or \textit{SR} is used, echoing Theorem 1 that $Pr(A)$ for \textit{CR} and \textit{SR} are equal for Win-by-One. They stay quite close for Win-by-Two.
		
		\item 	The expected length of a game, \textit{EL}, is always greater under Win-by-Two than under Win-by-One.  Whereas \textit{EL} for Win-by-Two never exceeds \textit{EL} for Win-by-One by more than one game under \textit{SR}, under \textit{CR} this difference may be two or more games, whether $p = \frac23$ or $\frac34$.  Clearly, the tiebreaker under Win-by-Two may significantly extend the average length of a game under \textit{CR}, rendering the game more competitive.
		
	\end{enumerate}
	
	\begin{table}[]
		\centering
		\scalebox{0.8}{
			\begin{tabular}{|ll||l|l|l|l|l|l|l|}
				\hline
				$p = \frac23$ & Best-of-$m$ & $Pr(A)$ & $Pr(B)$ & $Qr(A$) & $Qr(B)$ & $Pr(T)$ & $EL(WB1)$ & $EL(WB2)$ \\ \hline\hline
				& $m = 3 $    & 0.593 & 0.407 & 0.600 & 0.400 & 0.333 & 2.333   & 3.000   \\
				\textit{SR}     & $m = 11$    & 0.544 & 0.456 & 0.544 & 0.456 & 0.173 & 8.650   & 8.995   \\
				& $m = 21$    & 0.531 & 0.469 & 0.531 & 0.469 & 0.124 & 17.251  & 17.500  \\ \hline
				& $m = 3$     & 0.593 & 0.407 & 0.571 & 0.429 & 0.667 & 2.667   & 6.000   \\
				\textit{CR}     & $m = 11$    & 0.544 & 0.456 & 0.542 & 0.458 & 0.346 & 9.825   & 11.553  \\
				& $m = 21$    & 0.531 & 0.469 & 0.531 & 0.469 & 0.248 & 19.126  & 20.368  \\ \hline\hline
				$p = \frac34$ & Best-of-$m$ & $Pr(A)$ & $Pr(B)$ & $Qr(A)$ & $Qr(B)$ & $Pr(T)$ & $EL(WB1)$ & $EL(WB2)$ \\ \hline\hline
				& $m = 3$     & 0.656 & 0.344 & 0.667 & 0.333 & 0.250 & 2.250   & 2.667   \\
				\textit{SR}     & $m = 11$    & 0.573 & 0.427 & 0.574 & 0.426 & 0.139 & 8.240   & 8.472   \\
				& $m = 21$    & 0.551 & 0.449 & 0.552 & 0.448 & 0.101 & 16.540  & 16.708  \\ \hline
				& $m = 3 $    & 0.656 & 0.344 & 0.600 & 0.400 & 0.750 & 2.750   & 8.000   \\
				\textit{CR}     & $m = 11$    & 0.573 & 0.427 & 0.567 & 0.433 & 0.418 & 10.080  & 13.006  \\
				& $m = 21$    & 0.551 & 0.449 & 0.550 & 0.450 & 0.303 & 19.513  & 21.631  \\ \hline
		\end{tabular}}
		\vspace{2mm}
		\caption{Probability ($Pr$) that $A$ Wins and $B$ Wins in Best-of-$m$, and ($Qr$) that $A$ Wins and $B$ Wins in Best-of-$m$ with a tiebreaker, and that the tiebreaker is implemented (T), and Expected Length of a game (\textit{EL}), for $p = \frac23$ and $\frac34$.}
		\label{table3}
	\end{table}
	
	All the variable-rule service sports mentioned in Section 1---except for racquetball, which uses Win-by-One---currently use \textit{SR} coupled with Win-by-Two.  The normal winning score of a game in squash is 11 (Best-of-21) and of badminton is 21 (Best-of-41), but a tiebreaker comes into play when there is a tie at 10-10 (squash) or 20-20 (badminton).\footnote{In badminton, if there is still a tie at 29-29, a single ``golden point" determines the winner (see www.bwfbadminton.org).}  In volleyball, the winning score is usually 25; in a tiebreaker it is the receiving team, not the serving team, that is advantaged, because it can set up spike as a return, which is usually successful \cite{schilling2009}.
	
	The winning score in racquetball is 15, but unlike the other variable-rule service sports, a player scores points only when he or she serves, which prolongs games if the server is unsuccessful.  This scoring rule was formerly used in badminton, squash, and volleyball, lengthening games in each sport by as much as a factor of two.  The rule was abandoned because this prolongation led to problems in tournaments and discouraged television coverage \cite[p. 103]{barrow2012}.  We are not sure why it persists in racquetball, and why this is apparently the only service sport to use Win-by-One.
	
	We believe that all the aforementioned variable-rule service sports would benefit from \textit{CR}.  Games would be extended and, hence, be more competitive, without significantly altering win probabilities. Specifically, \textit{CR} gives identical win probabilities to \textit{SR} under Win-by-One, and very similar win probabilities under Win-by-Two.
	
	\section{The Fixed Rules of Table Tennis and Tennis}
	
	Both table tennis and tennis use Win-by-Two, but neither uses \textit{SR}.  Instead, each uses a fixed rule.  In table tennis, the players alternate, each serving on two consecutive points, independent of the score and of who wins any point.  The winning score is 11 unless there is a 10-10 tie, in which case there is a Win-by-Two tiebreaker, in which the players alternate, but serve on one point instead of two.
	
	In table tennis, there seems to be little or no advantage to serving, so it does not matter much which player begins a tiebreaker. If $p = \frac12$, the playing field is level, so neither player gains a probabilistic advantage from being the first double server.
	
	Tennis is a different story.\footnote{Part of this section is adapted from \cite{brams2016}.}  It is generally acknowledged that servers have an advantage, perhaps ranging from about $p = \frac35$ to $p = \frac34$ in a professional match. In tennis, points are organized into games, games into sets, and sets into matches.  Win-by-Two applies to games and sets.  In the tiebreaker for sets, which occurs after a 6-6 tie in games, one player begins by serving once, after which the players alternate serving twice in a row.\footnote{Arguably, the tiebreaker creates a balance of forces: $A$ is advantaged by serving at the outset, but $B$ is then given a chance to catch up, and even move ahead, by next having two serves in a row. For a recent empirical study on this sequence, see \cite{cohen-zada2018}.}
	
	A tennis tiebreaker begins with one player---for us, $A$---serving.  Either $A$ wins the point or $B$ does.  Regardless of who wins the first point, there is then a fixed alternating sequence of double serves, $BBAABB\ldots$, for as long as is necessary (see the rule for winning in the next paragraph).  When $A$ starts, the entire sequence can be viewed as one of two alternating single serves, broken by the slashes shown below,
	\[ AB/BA/AB/BA \ldots.\]
	
	\noindent Between each pair of adjacent slashes, the order of $A$ and $B$ switches as one moves from left to right.  We call the serves between adjacent slashes a block.
	
	The first player to score 7 points, and win by a margin of at least two points, wins the tiebreaker.  Thus, if the players tie at 6-6, a score of 7-6 is not winning.  In this case, the tiebreaker would continue until one player goes ahead by two points (e.g., at 8-6, 9-7, etc.).
	
	Notice that, after reaching a 6-6 tie, a win by one player can occur only after the players have played an even number of points, which is at the end of an $AB$ or $BA$ block.  This ensures that a player can win only by winning twice in a block---once on the player's own serve, and once on the opponent's. The tiebreaker continues as long as the players continue to split blocks after a 6-6 tie, because a player can lead by two points only by winning both serves in a block.  When one player is finally ahead by two points, thereby winning the tiebreaker, the players must have had exactly the same number of serves.
	
	The fixed order of serving in the tennis tiebreaker, which is what precludes a player from winning simply because he or she had more serves, is not the only fixed rule that satisfies this property.  The strict alternation of single serves,
	\[ AB/AB/AB/AB \ldots, \]
	or what Brams and Taylor \cite{brams1999} call ``balanced alternation,"
	\[ AB/BA/BA/AB \ldots, \]
	are two of many alternating sequences that create adjacent $AB$ or $BA$ blocks.\footnote{Brams and Taylor \cite[p. 38]{brams1999} refer to balanced alternation as ``taking turns taking turns taking turns $\ldots$"  This sequence was proposed and analyzed by several scholars and is also known as the Prouhet--Thue--Morse (PTM) sequence \cite[pp. 82--85]{palacios2014}.  Notice that the tennis sequence maximizes the number of double repetitions when written as $A/BB/AA/BB/ \ldots$, because after the first serve by one player, there are alternating double serves by each player.  This minimizes changeover time and thus the ``jerkiness" of switching servers. }  All are fair---they ensure that the losing player does not lose only because he or she had fewer serves---for the same reason that the tennis sequence is fair.\footnote{It is true that if $A$ serves first in a block, he may win after serving one more time than $B$ if the tiebreaker does not go to a 6-6 tie. For example, assume that the score in the tiebreaker is 6-4 in favor of $A$, and it is $A$'s turn to serve.  If $A$ wins the tiebreaker at 7-4, he will have had one more serve than $B$.  However, his win did not depend on having one more serve, because he is now 3 points ahead.  By comparison, if the tiebreaker had gone to 6-6, $A$ could not win at 7-6 with one more serve; his win in this case would have to occur at the end of a block, when both players have had the same number of serves.}
	
	Variable-rule service sports, including badminton, squash, racquetball, and volleyball, do not necessarily equalize the number of times the two players or teams serve.  Under \textit{SR}, if $A$ holds his serve throughout a game or a tiebreaker, he can win by serving all the time.
	
	This cannot happen under \textit{CR}, because $A$ loses his serve when he wins; he can hold his serve only when he loses.  But this is not to say that $A$ cannot win by serving more often.  For example, in Best-of-3 \textit{CR}, $A$ can beat $B$ 2-1 by serving twice.  But the tiebreaker rule in tennis, or any other fixed rule in which there are alternating blocks of $AB$ and $BA$, does ensure that the winner did not benefit by having more serves.
	
	\section{Summary and Conclusions}
	
	We have analyzed four rules for service sports---including the Standard Rule (\textit{SR}) currently used in many service sports---that make serving variable, or dependent on a player's (or team's) previous performance.  The three new rules we analyzed all give a player, who loses a point or falls behind in a game, the opportunity to catch up by serving, which is advantageous in most service sports.
	The Catch-Up Rule (\textit{CR}) gives a player this opportunity if he or she just lost a point---instead of just winning a point, as under \textit{SR}.  Each of the Trailing Rules (\textit{TR}s) makes the server the player who trails; if there is a tie, the server is the player who previously was ahead (\textit{TRa}) or behind (\textit{TRb}).
	
	For Win-by-One, we showed that \textit{SR} and \textit{CR} give the players the same probability of a win, independent of the number of points needed to win.  We proved, and illustrated with numerical calculations, that the expected length of a game is greater under \textit{CR} than under \textit{SR}, rendering it more likely to stay close to the end.
	
	By contrast, the two \textit{TR}s give the player who was not the first server, $B$, a greater probability of winning, making it less likely that they would be acceptable to strong players, especially those who are used to \textit{SR} and have done well under it.  In addition, \textit{TRb} is not strategy-proof: We exhibited an instance in which a player can benefit by deliberately losing, which also makes it less appealing.
	
	We analyzed the effects of Win-by-Two, which most service sports currently combine with \textit{SR}, showing that it is compatible with \textit{CR}.  We showed that the expected length of a game, especially under \textit{CR}, is always greater under Win-by-Two than Win-by-One.
	
	On the other hand, Win-by-Two, compared with Win-by-One, has little effect on the probability of winning (compare $Qr(A)$ and $Qr(B)$ with $Pr(A)$ and $Pr(B)$ in Table 3).\footnote{\textit{CR} would make the penalty shootout of soccer fairer by giving the team that loses the coin toss---and which, therefore, usually must kick second in the shootout---an opportunity to kick first on some kicks \cite{brams2016}.}  The latter property should make \textit{CR} more acceptable to the powers-that-be in the different sports who, generally speaking, eschew radical changes that may have unpredictable consequences.  But they want to foster competitiveness in their sports, which \textit{CR} does.
	
	Table tennis and tennis use fixed rules for serving, which specify when the players serve and how many serves they have.  We focused on the tiebreaker in tennis, showing that it was fair in the sense of precluding a player from winning simply as a result of having served more than his or her opponent.  \textit{CR} does not offer this guarantee in variable-rule sports, although it does tend to equalize the number of times that each player serves.
	
	There is little doubt that suspense is created, which renders play more exciting and unpredictable, by making who serves next dependent on the success or failure of the server---rather than fixing in advance who serves and when.  But a sport can still generate keen competition, as the tennis tiebreaker does, even with a fixed service schedule. Thus, the rules of tennis, and in particular the tiebreaker, seem well-chosen; they create both fairness and suspense.
	
	\section{Appendix}
	
	\subsection{Theorem 1: \textit{SR} and \textit{CR} Give Equal Win Probabilities}
	We define a serving schedule for a Best-of-$(2k+1)$ game to be
	\[(A_1, \ldots, A_{k + 1}, B_1, \ldots, B_k) \in \{W, L\}^{2k+1},\]
	where for $i = 1, 2, \ldots, k+1$, $A_i$ records the result of $A$'s $i$th serve, with $W$ representing a win for $A$ and $L$ a loss for $A$, and for $j = 1, 2, \ldots, k, B_j$ represents the result of $B$'s $j$th serve, where now $W$ represents a win for $B$ and $L$ represents a loss for $B$.
	
	The order in which random variables are determined does not influence the result of a game.  So we may as well let $A$ and $B$ determine their service schedule in advance, and then play out their game under \textit{SR} or \textit{CR}, using the predetermined result of each serve.
	
	The auxiliary rule, \textit{AR}, is a serving rule for a Best-of-$(2k+1)$ game in which $A$ serves $k+1$ times consecutively, and then $B$ serves $k$ times consecutively. For convenience, we can assume that \textit{AR} continues for $2k+1$ serves, even after one player accumulates $k+1$ points. The basis of our proof is a demonstration that, if the serving schedule is fixed, then games played under all three serving rules, \textit{AR}, \textit{SR}, and \textit{CR}, are won by the same player.
	
	\begin{theorem}
		Let $k \geq 1$.  In a Best-of-$(2k+1)$ game, $Pr_{SR}(A) = Pr_{CR}(A)$.
	\end{theorem}
	
	\begin{proof}
		Suppose we have already demonstrated that, for any serving schedule, the three service rules---\textit{AR}, \textit{SR}, and \textit{CR}---all give the same winner.  It follows that the subset of service schedules under which $A$ wins under \textit{SR} must be identical to the subset of service schedules under which $A$ wins under \textit{CR}.  Moreover, regardless of the serving rule, any serving schedule that contains $n$ wins for $A$ as server and $m$ wins for $B$ as server must be associated with the probability $p^n(1-p)^{k+1-n}q^m(1-q)^{k-m}$.%
		\footnote{What really matters is that any serving schedule is associated with some probability. For example, as discussed in footnote 4, if  $A$ wins his $i$th serve with probability $p_i$, for $i=1,2,\ldots$, and that $B$ wins her $j$th serve with probability $q_j$, for $j=1,2,\ldots$, then the formula would be more complicated, but the probability of choosing a particular service schedule would still be independent of the order of serves. Therefore, the proof would apply to the general model described in footnote 4.}
		Because the probability that a player wins under a service rule must equal the sum of the probabilities of all the service schedules in which the player wins under that rule, the proof of the theorem will be complete.
		
		Fix a serving schedule $(A_1, \ldots, A_{k + 1}, B_1, \ldots, B_k) \in \{W, L\}^{2k+1}$.  Let $a$ be the number of $A$'s server losses and $b$ be the number of $B$'s server losses. (For example, for $k = 2$ and serving schedule $(W, L, L)$, $a = 1$ and $b = 1$.)
		
		Under service rule \textit{AR}, there are $2k+1$ serves in total. Then $A$ must accumulate $k+1+b-a$ points and $B$ must accumulate $k-b+a$ points.  Clearly, $A$ has strictly more points than $B$ if and only if $b \geq a$.
		
		Now consider service rule \textit{SR}, under which service switches whenever the server loses.  First we prove that, if $b \geq a$, $A$ will have the opportunity to serve at least $k+1$ times. $A$ serves until he loses, and then $B$ serves until she loses, so immediately after $B$'s first loss, $A$ has also lost once and is to serve next.  Repeating, immediately after $B$'s $a$th loss, $A$ has also lost $a$ times and is to serve next.  Either $A$'s prior loss was his $(k+1)$st serve, or $A$ wins every serve from this point on, including $A$'s $(k+1)$st serve.  After this serve, $A$ has $k+1-a+b \geq k+1$ points, so $A$ must have won either on this serve or earlier.
		
		Now suppose that $a > b$ under service rule \textit{SR}.  Then, after $B$'s $b$th loss, $A$ has also lost $b$ times and is to serve.  Moreover, $A$ must lose again (since $b < a$), and then $B$ becomes server and continues to serve (without losing) until $B$ has served $k$ times.  By then, $B$ will have gained $k-b+a \geq k+1$ points, so $B$ will have won. In conclusion, under \textit{SR}, $A$ wins if and only if $b \geq a$.
		
		Consider now service rule \textit{CR}, under which $A$ serves until he wins, then $B$ serves until she wins, etc. Clearly, $B$ cannot win a point on her own serve until after $A$ has won a point on his own serve.  Repeating, if $B$ has just won a point on serve, then $A$ must have already won the same number of points on serve as $B$.
		
		Again assume that $b \geq a$.  Note that $A$ has $k+1-a$ server wins in his first $k+1$ serves, $B$ has $k-b$ server wins in her first $k$ serves, and $k+1-a > k-b$.  Consider the situation under \textit{CR} immediately after $B$'s $(k-b)$th server win. (All of $B$'s serves after this point up to and including the $k$th serve must be losses.)  At this point, both $A$ and $B$ have $k-b$ server wins.  Suppose that $A$ also has $a^\prime \leq a$ server losses, and that $B$ has $b^\prime \leq  b$ server losses.  Then the score at this point must be $(k-b+b^\prime, k-b+a^\prime)$.
		
		Immediately after $B$'s $(k-b)$th server win, $A$ is on serve.  For $A$, the schedule up to $A$'s $(k+1)$st serve must contain $a-a^\prime$ server losses and $k+1-a-(k-b) = b-a+1 > 0$ server wins.  Immediately after $A$'s next server win, which is $A$'s $(k-b+1)$st server win, $A$'s score is $k-b+b^\prime +1$, and $B$'s score is at most $k-b+a^\prime + (a-a^\prime) = k-b+a  \leq  k$.  Then $B$ is on serve, and all of $B$'s $b-b^\prime$ remaining serves (up to and including $B$'s $k$th serve) must be server losses.  Therefore, after $B$'s $k$th serve, $A$'s score is $k-b+b^\prime+1+ (b-b^\prime) = k+1$, and $A$ wins.
		
		If $b <a$, an analogous argument shows that $B$ wins under \textit{CR}.  This completes the proof that, under any serving schedule, the winner under \textit{AR} is the same as the winner under \textit{SR} and under \textit{CR}.
	\end{proof}

	\subsection{Theorem 2: Expected Lengths under \textit{SR} and \textit{CR}}
	\begin{lemma}
		Let $1 \leq t \leq s \leq r$. If a subset of $s$ dots is selected uniformly from a sequence of $r$ dots, the expected position of the $t$th selected dot in the full sequence is
		\[ t\left(\frac{r+1}{s+1}\right). \]
	\end{lemma}
	
	\begin{proof}
		If $i$ is any nonselected dot, let $X_i$ be an indicator variable such that $X_i = 1$ if dot $i$ precedes the $t$th selected dot, and $X_i = 0$ otherwise. Then $E[X_i] = P(X_i = 1)$ and the expected position of the $t$th dot is $E[\sum_i X_i ] + t$, where the sum is taken over all nonselected dots.
		
		Now $P(X_i = 1)$ depends only on the position of dot $i$ and the $s$ selected dots; considering only these $s + 1$ dots, $X_i = 1$ if and only if dot $i$ is in the first, second, $\ldots$, or $t$th position. In other words, for any nonselected dot $i$, $P(X_i = 1) = \frac {t}{s+1}$. Therefore, the expected position of the $t$th dot is $(r - s)\frac {t}{s+1} + t = t \frac{r + 1}{s+1}$, as required.
	\end{proof}
	
	\begin{theorem}
		If $0 < p < 1$, $0 < q < 1$, and $k\geq 1$, then the expected length of Best-of-$(2k+1)$ game under \textit{CR} is greater than under \textit{SR} if and only if $p+q>1$.
	\end{theorem}
	
	\begin{proof}
		
		As defined earlier, a service schedule is a sequence of exactly $k+1$ wins and losses on $A$'s serves and exactly $k$ wins and losses on $B$'s serves. A service schedule has parameters $(n,m)$ if it contains $n$ server wins for $A$ and $m$ server wins for $B$. The probability of any particular service schedule with parameters $(n,m)$ is
		\[ Pr_0(n,m) = p^n (1-p)^{k+1-n} q^m (1-q)^{k-m}.\]
		The probability that \textit{some} service schedule with parameters $(n,m)$ occurs is
		\[ Pr(n,m) = \binom{k+1}{n} \binom{k}{m} p^n (1-p)^{k+1-n} q^m (1-q)^{k-m}. \]
		
		Consider a service schedule with parameters $(n,m)$. Using \textit{SR}, if $n>m$, then $A$ wins and exhausts his part of the service schedule.  $B$ uses her part of the schedule through her $(k+1-n)$th loss. From Lemma 1 with $r=k$, $s=k-m$, $t=k+1-n$, the expected length of the game is
		\begin{align*}
		EL_{SR}^0(n,m) &= k+1 + (k+1-n)\left( \frac{k+1}{k+1-m} \right) \\
		&= 2(k+1)-(n-m) \left( \frac{m}{k+1-m}+1 \right).
		\end{align*}
		\noindent However, if $n\leq m$, then $B$ wins and exhausts her part of the service schedule, while $A$ uses his part of the schedule through his $(k+1-m)$th loss. From Lemma 1 with $r=k+1$, $s=k+1-n$, $t=k+1-m$, the expected length of the game is
		\begin{align*}
		EL_{SR}^0(n,m) &= k + (k+1-m)\left(\frac{k+1+1}{k+1-n+1}\right)\\
		&= 2(k+1)-(m-n+1) \left( \frac{n}{k+1-n+1}+1 \right).
		\end{align*}
		
		Using \textit{CR}, if $n>m$, then $A$ wins, $B$ exhausts her part of the service schedule, and $A$ uses his schedule through his $(m+1)$st win. From Lemma 1 with $r=k+1$, $s=n$, $t=m+1$, the expected length of the game is
		\begin{align}
		EL_{CR}^0(n,m) &= k + (m+1)\left(\frac{k+1+1}{n+1}\right) \notag\\
		&= 2(k+1)-(n-m) \left(\frac{k+1-n}{n+1}+1 \right). \notag
		\end{align}
		But if $n\leq m$, then $B$ wins, $A$ exhausts his part of the service schedule, and $B$ uses her part through her $n$th win. From Lemma 1 with $r=k$, $s=m$, $t=n$, the expected length of the game is
		\begin{align*}
		EL_{CR}^0(n,m) &= k+1 + n \left( \frac{k+1}{m+1} \right)\\
		&= 2(k+1)-(m-n+1) \left( \frac{k-m}{m+1}+1 \right).
		\end{align*}
		
		Combining the first two formulas, we can write the expected length of a game played under \textit{SR}, \footnote{We take the sum $\sum_{n>m}$ to include terms for every integer pair $(n,m)$ with $n>m$.  If $n>k+1$ or $m>k$, we take $Pr(n,m)=0$, so the sum includes only finitely many nonzero terms.  Similarly, the sum $\sum_{n\leq m}$ includes terms for every integer pair $(n,m)$ with $n\leq m$.  In general, we interpret the binomial coefficient $\binom{a}{b}$ to be zero whenever $b<0$ or $b>a$.}
		\begin{multline*}
		EL_{SR} = \sum_{n>m} Pr(n,m) \left[ 2(k+1)-(n-m)\left(\frac{m}{k+1-m}+1\right) \right] \\
		+ \sum_{n\leq m} Pr(n,m) \left[ 2(k+1)-(m-n+1)\left(\frac{n}{k+1-n+1}+1\right) \right],
		\end{multline*}
		and the expected length of a game played under \textit{CR},
		\begin{multline*}
		EL_{CR} = \sum_{n>m} Pr(n,m) \left[ 2(k+1)-(n-m)\left(\frac{k+1-n}{n+1}+1\right) \right] \\
		+\sum_{n\leq m} Pr(n,m) \left[ 2(k+1)-(m-n+1)\left(\frac{k-m}{m+1}+1\right) \right].
		\end{multline*}
		When we subtract these expressions, many terms cancel:
		\begin{align*}
		EL_{CR}-EL_{SR}  &= \sum_{n>m} Pr(n,m) (n-m) \left[\frac{m}{k+1-m}-\frac{k+1-n}{n+1}\right]\\
		&~~+ \sum_{n\leq m} Pr(n,m) (m-n+1) \left[ \frac{n}{k+1-n+1}-\frac{k-m}{m+1} \right].
		\end{align*}
		Our objective now is to show that this quantity, the expected number of points by which the length of a \textit{CR} game exceeds the length of an \textit{SR} game, has the same sign as $p+q-1$.
		
		To simplify this expected difference in lengths, we extract the binomial coefficients $\binom{k+1}{n}$ and $\binom{k}{m}$ from the probability factors and use the following identities:
		\begin{align*}
		&\binom{k}{m} \cdot \frac{m}{k+1-m}     = \binom{k}{m-1}; &\binom{k+1}{n}   \cdot \frac{k+1-n}{n+1}   = \binom{k+1}{n+1}; \\
		&\binom{k+1}{n}   \cdot \frac{n}{k+1-n+1}   = \binom{k+1}{n-1}; &\binom{k}{m} \cdot \frac{k-m}{m+1} = \binom{k}{m+1}.
		\end{align*}
		The result is
		\begin{align*}
		&EL_{CR}-EL_{SR}\\
		&~= \sum_{n>m} Pr_0(n,m) (n-m) \left[\binom{k+1}{n}\binom{k}{m-1}-\binom{k+1}{n+1}\binom{k}{m}\right] \\ 
		&~~~+\sum_{n\leq m} Pr_0(n,m) (m-n+1) \left[ \binom{k+1}{n-1}\binom{k}{m}-\binom{k+1}{n}\binom{k}{m+1} \right].
		\end{align*}
		Now split the sums to obtain
		\begin{align*}
		EL_{CR}-EL_{SR} =& \sum_{n>m} Pr_0(n,m) (n-m) \binom{k+1}{n}\binom{k}{m-1} \\
		&- \sum_{n>m} Pr_0(n,m) (n-m)\binom{k+1}{n+1}\binom{k}{m} \\
		&~+ \sum_{n\leq m} Pr_0(n,m) (m-n+1) \binom{k+1}{n-1}\binom{k}{m} \\
		&~~-\sum_{n\leq m} Pr_0(n,m) (m-n+1) \binom{k+1}{n}\binom{k}{m+1}.
		\end{align*}
		
		To analyze the expression for $EL_{CR}-EL_{SR}$, we first shift indices in the second summation by  replacing $n$ by $n-1$ and $m$ by $m-1$ throughout. Of course, the condition $n>m$ and the factor $n-m$ are not affected. Then
		\begin{align*}
		&\sum_{n>m} Pr_0(n,m)(n-m)\binom{k+1}{n+1}\binom{k}{m} \\
		&~= \sum_{n>m} Pr_0(n-1,m-1) (n-m) \binom{k+1}{n}\binom{k}{m-1}\\
		&~= \frac{1-p}{p}\,\frac{1-q}{q} \sum_{n>m} Pr_0(n,m) (n-m)\binom{k+1}{n}\binom{k}{m-1} ,\\
		\end{align*}
		
		\vspace{-0.3 in}
		
		\noindent where in the last step we have used the relationship
		\[ Pr_0(n-1,m-1)= \frac{1-p}{p}\,\frac{1-q}{q}\ Pr_0(n,m)\]
		to make the second summation match the first. Similarly, we work on the fourth summation, shifting the indices in the same way---replacing $n$ by $n-1$ and $m$ by $m-1$ throughout. Again, the condition $n\leq m$ and the factor $m-n+1$ are not affected, and in the end the fourth summation matches the third.
		\begin{align*}
		& \sum_{n\leq m} Pr_0(n,m)(m-n+1) \binom{k+1}{n}\binom{k}{m+1}\\
		&~=  \sum_{n\leq m} Pr_0(n-1,m-1) (m-n+1) \binom{k+1}{n-1}\binom{k}{m}\\
		&~= \frac{1-p}{p}\,\frac{1-q}{q} \sum_{n\leq m} Pr_0(n,m) (m-n+1) \binom{k+1}{n-1}\binom{k}{m}.
		\end{align*}
		
		Incorporating these changes into our expression for $EL_{CR}-EL_{SR}$ gives
		\begin{align*}
		&EL_{CR}-EL_{SR} \\
		&~= \left(1-\frac{1-p}{p}\,\frac{1-q}{q}\right) \sum_{n>m} Pr_0(n,m) (n-m) \binom{k+1}{n}\binom{k}{m-1}\\
		&~~~+ \left(1-\frac{1-p}{p}\,\frac{1-q}{q}\right) \sum_{n\leq m} Pr_0(n,m) (m-n+1)  \binom{k+1}{n-1}\binom{k}{m}.
		\end{align*}
		This completes the proof, because the sums include only positive terms, and the common factor is
		\[
		1-\frac{1-p}{p}\,\frac{1-q}{q} = \frac{p+q-1}{pq},
		\]
		which has the same sign as $p+q-1$.
	\end{proof}
	
	\subsection{Theorem 3: Strategy-proofness of \textit{TRa}}
	\begin{theorem}
		\textit{TRb} is strategy-vulnerable, whereas \textit{SR} and \textit{CR} are strategy-proof. \textit{TRa} is strategy-proof whenever $p+q>1$.
	\end{theorem}
	
	\vspace{2mm}
	See Section 2 for the proof that \textit{TRb} is strategy-vulnerable and that \textit{SR} and \textit{CR} are strategy-proof. Here we consider only a Best-of-$(2k+1)$ game played under \textit{TRa}, and use notation similar to the text: $(C, x, y)$ denotes a state in which player $C$ ($C =A \text{ or } B$) is about to serve, $A$'s current score is $x$, and $B$'s is $y$. Let $W_{A}(C, x, y)$ denote $A$'s conditional win probability given that the game reaches state $(C, x, y)$. Assume that $p + q > 1$.
	
	Given any state, call the state that was its immediate predecessor its \emph{parent}. Call two states \emph{siblings} if they have the same parent. For example, state $(A, 0, 0)$ is the parent of siblings $(B, 1, 0)$ and $(A, 0, 1)$. Let $0 < x \leq k + 1$ and $0 \leq y \leq k + 1$ with $x + y \leq 2k + 1$. Then any state $(C, x, y)$ must have a sibling $(C^\prime, x - 1, y + 1)$, namely the state that would have arisen had $B$, and not $A$, won the last point.
	
	To show that the Best-of-$(2k+1)$ game played under \textit{TRa} is strategy-proof for $A$, we must show that $W_{A}(C, x, y) \geq W_{A}(C^\prime, x - 1, y + 1)$ whenever $(C, x, y)$ and $(C^\prime, x - 1, y + 1)$ are siblings. For purposes of induction, note that, if $x$ and $y$ satisfy $0 \leq x \leq k + 1$, $0 \leq y \leq k + 1$, and $x + y \leq 2k + 1$, then $W_{A}(C, x, y) = 1$ if $x = k + 1$ and $W_{A}(C, x, y) = 0$ if $y = k + 1$, where $C =A \text{ or } B$. In these cases, $(C, x, y)$ is a \emph{terminal state} and can be written $(x, y)$, as there are no more serves.
	
	\vspace{2mm}
	\noindent \textbf{Lemma 2.} \emph{Suppose that $x \not =y$ and the state is $(C, x, y)$. Then either $x < y$ and $C = A$, or $x > y$ and $C = B$.}
	\begin{proof} Under \textit{TRa}, the player about to serve must be the player with the lower score. \end{proof}
	
	\vspace{2mm}
	\noindent \textbf{Lemma 3.} \emph{The states $(A, x, y)$ and $(B, x, y)$ can both arise if and only if $0 < x = y < k + 1$. In this case, the state is $(A, x, x)$ if the parent was $(B, x, x - 1)$, and the state is $(B, x, x)$ if the parent was $(A, x - 1, x)$.}
	
	\begin{proof} The first statement follows from Lemma 2. The second is a paraphrase of \textit{TRa}.
	\end{proof}
	
	Note that Lemma 3 fails for \textit{TRb}; in fact, it captures the difference between \textit{TRa} and \textit{TRb}. For example, under \textit{TRb}, if $A$ loses at $(B, 1, 0)$, the next state will be $(B, 1, 1)$ (it would be $(A, 1, 1)$ under \textit{TRa}), and if $A$ wins at $(A, 0, 1)$, the next state will be $(A, 1, 1)$ (rather than $(B, 1, 1)$ under \textit{TRa}).
	
	\vspace{2mm}
	\noindent \textbf{Lemma 4.} \emph{$W_A(A, x, x) > W_A(B, x, x)$ if and only if $$W_A(B, x+1, x) > W_A(A, x, x+1).$$}
	
	\begin{proof} First notice that $W_A(A, x, x) = pW_A(B, x + 1, x) + (1-p)W_A(A, x, x+1)$ and $W_A(B, x, x) = (1-q)W_A(B, x+1, x) + qW_A(A, x, x+1)$. Therefore
		\[W_A(A, x, x) - W_A(B, x, x) = (p + q - 1) \left[W_A(B, x+1, x)- W_A(A, x, x+1) \right]\]
		and the claim follows from the assumption that $p + q > 1$.\end{proof}
	
	\begin{figure}
		\begin{center}
			\includegraphics[width=0.88\textwidth]{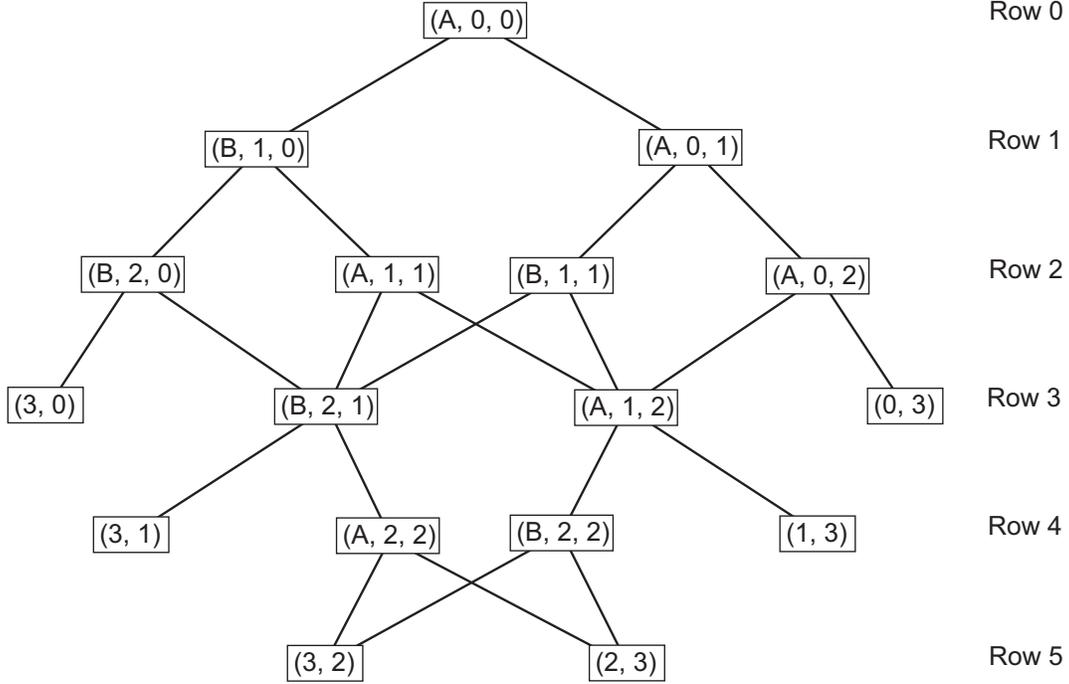}
			\caption{Best-of-5 \textit{TRa} probability tree.}
			\label{fig3}
		\end{center}
	\end{figure}
	
	We now consider the probability tree of a Best-of-$(2k + 1)$ game played under \textit{TRa}. The Best-of-5 tree ($k = 2$) is shown in Figure 3. The nodes (states) are labeled $(C, x, y)$, where $C$ is the player to serve, $x$ is the score of $A$, and $y$ is the score of $B$. Because there are no servers after the game has been won, the terminal nodes are labeled either $(k + 1, y)$ or $(x, k + 1)$. The probability tree is rooted at $(A, 0, 0)$ and is binary---every nonterminal node is parent to exactly two nodes. For nonterminal nodes where $A$ serves, the left-hand outgoing arc has probability $p$ and the right-hand outgoing arc has probability $1 - p$; for nonterminal nodes where $B$ serves, the two outgoing arcs have probability $1 - q$ and $q$, respectively.
	
	The probability tree for the Best-of-$(2k + 1)$ game played under \textit{TRa} has $2k + 2$ rows, numbered 0 (at the top) to $2k + 1$ (at the bottom). The top row contains only the initial state, $(A, 0, 0)$, and the bottom row contains only terminal nodes. At every node in row $\ell$, the players' scores sum to $\ell$.
	
	In the upper half of the tree, above row $k + 1$, there are no terminal nodes. Let $1 \leq \ell < k + 1$. If $\ell$ is odd, every state in row $\ell$ is of the form $(C, x, \ell - x)$. By Lemma 2, there are $\ell + 1$ states in row $\ell$, from $(B, \ell, 0)$ to $(A, 0, \ell)$. If $\ell$ is even, Lemmata 2 and 3 show that row $\ell$ contains $\ell + 2$ states, from $(B, \ell, 0)$ to $(A, 0, \ell)$, including both $(A, \frac\ell2, \frac\ell2)$ and $(B, \frac\ell2, \frac\ell2)$. We place $(A, \frac\ell2, \frac\ell2)$ to the left of $(B, \frac\ell2, \frac\ell2)$.
	
	In the lower half of the tree---row $k + 1$ and below---each row begins and ends with a terminal node. Let $k + 1 \leq \ell \leq 2k + 1$. Then the first entry in row $\ell$ is the terminal node $(k + 1, \ell - k - 1)$, and the last entry is the terminal node $(\ell - k - 1, k + 1)$. If $\ell$ is odd, every nonterminal node in row $\ell$ is a state of the form $(C, x, \ell - x)$ for $x = k, k - 1, \ldots, \ell - k$. (If $\ell = 2k + 1$, the only nodes in row $\ell$ are two terminal nodes.) By Lemma 2, row $\ell$ contains $2k - \ell + 3$ nodes in total. If $\ell$ is even, row $\ell$ contains $2k - \ell + 4$ nodes, from the terminal node $(k, \ell - k)$ to the terminal node $(\ell - k, k)$, including all possible states of the form $(C, x, \ell - x)$ for $x = k, k - 1, \ldots, \ell - k$. Among these states are both $(A, \frac\ell2, \frac\ell2)$ and $(B, \frac\ell2, \frac\ell2)$, with the former state on the left.
	
	As the figure illustrates, sibling states have a unique parent unless they are of the form $(B, x + 1, x)$ and $(A, x, x + 1)$, in which case both $(A, x, x)$ and $(B, x, x)$ are parents. Thus, $(B, x + 1, x)$ and $(A, x, x + 1)$ could be called ``double siblings."
	
	The method of proof is to show by induction that the function $W_A(\cdot)$ is decreasing on each row as one reads from left to right. This will prove that \textit{TRa} is strategy-proof for $A$, because it will show that, in every state, $A$ does better by winning the next point rather than losing it to obtain the sibling state. Since $W_B(\cdot) = 1 - W_A(\cdot)$, the proof also shows that $W_B(\cdot)$ is increasing on each row, and therefore that \textit{TRa} is strategy-proof for $B$.
	
	\begin{proof}To begin the induction, observe that $W_A(\cdot)$ is decreasing on row $2k + 1$ since $W_A(k + 1, k) = 1$ and $W_A(k, k + 1) = 0$. Now consider row $2k$, which begins with a terminal node $(k + 1, k - 1)$, where $W_A(k + 1, k-1) = 1$, and ends with a terminal node $(k - 1, k + 1)$, where $W_A(k-1, k+ 1) = 0$. Row $2k$ contains 4 nodes; its second entry is $(A, k, k)$ and its third is $(B, k, k)$. To apply Lemma 4 with $x = k$, note that $(B, k + 1, k)$ and $(A, k, k + 1)$ are the terminal nodes $(k + 1, k)$ and $(k, k + 1)$, and $W_A(k+ 1, k) = 1 > W_A(k, k+ 1) = 0$. Lemma 4 now implies that $W_A(A, k, k) > W_A(B, k, k)$, as required.
		
		Now consider any row $\ell$, and assume that $W_A(\cdot)$ has been shown to be strictly decreasing on row $\ell + 1$. If $\ell \geq k + 1$, row $\ell$ begins and ends with a terminal node, for which $W_A(k + 1, \ell - k - 1) = 1$ and $W_A(\ell - k - 1, k + 1) = 0$; any other node in row $\ell$ is a nonterminal node. Suppose that $(C, x, \ell - x)$ and $(C^\prime, x^\prime, \ell - x^\prime)$ are adjacent nodes in row $\ell$ and that $(C, x, \ell - x)$ is on the left. First suppose that $x = x^\prime$. Then by Lemma 3, $\ell$ must be even, and the two nodes must be $(A, \frac\ell2, \frac\ell2)$ (on the left) and $(B, \frac\ell2, \frac\ell2)$ (on the right). The induction assumption and Lemma 4 now implies that $W_A(A, \frac\ell2, \frac\ell2) > W_A(B, \frac\ell2, \frac\ell2)$.
		
		Otherwise, consecutive nodes $(C, x, \ell - x)$ and $(C^\prime, x^\prime, \ell - x^\prime)$ in row $\ell$ must satisfy $x^\prime = x - 1$, by Lemmata 2 and 3. Thus we are comparing $W_A(C, x, \ell - x)$ with $W_A(C^\prime, x - 1, \ell - x + 1)$. Now
		\begin{align*}
		W_A(C, x, \ell - x) &= r W_A(D, x + 1, \ell - x) + (1 - r)W_A(D^\prime, x, \ell - x + 1) \\
		&> W_A(D^\prime, x, \ell - x + 1),\\
		W_A(C^\prime, x - 1, \ell& - x + 1)\\
		&~= s W_A(E, x, \ell - x + 1) + (1 - s)W_A(E^\prime, x - 1, \ell - x + 2) \\
		&~> W_A(E, x, \ell - x + 1),
		\end{align*}
		where $0 < r, s < 1$ because each of $r$ and $s$ must equal one of $p$, $q$, $1 - p$, and $1 - q$.
		
		According to Lemma 3, $(D^\prime, x, \ell - x + 1) = (E, x, \ell - x + 1)$ unless $x = \ell - x + 1$, i.e., $x = \frac {\ell + 1}{2}$ (which of course requires that $\ell$ be odd). First assume that $x \not= \frac {\ell + 1}{2}$. Then we have shown that
		\begin{align*}
		W_A(C, x, \ell - x) &> W_A(D^\prime, x, \ell - x + 1)\\
		&= W_A(E, x, \ell - x + 1) > W_A(C^\prime, x - 1, \ell - x + 1),
		\end{align*}
		as required.
		
		Now assume that $x = \frac {\ell + 1}{2}$. Then the original states $(C, x, \ell - x)$ and $(C^\prime, x - 1, \ell - x + 1)$ must have been $(B, \frac{\ell + 1}{2}, \frac{\ell - 1}{2})$ and $(A, \frac{\ell - 1}{2}, \frac{\ell + 1}{2})$. Moreover, $(D^\prime, x, \ell - x + 1) = (A, \frac {\ell + 1}{2}, \frac {\ell + 1}{2})$ and $(E, x, \ell - x + 1) = (B, \frac{\ell + 1}{2}, \frac{\ell + 1}{2})$. Recall that $(A, x, x)$ always appears to the left of $(B, x, x)$. Thus the assumption that $W_A(\cdot)$ is decreasing on row $\ell + 1$ implies that $W_A(A, \frac {\ell + 1}{2}, \frac {\ell + 1}{2}) > W_A(B, \frac {\ell + 1}{2}, \frac {\ell + 1}{2})$. By Lemma 4, we have shown that
		\begin{align*}
		W_A(B, \frac{\ell + 1}{2}, \frac{\ell - 1}{2}) &> W_A(A, \frac {\ell + 1}{2}, \frac {\ell + 1}{2})\\
		&> W_A(B, \frac {\ell + 1}{2}, \frac {\ell + 1}{2}) > W_A(A, \frac{\ell - 1}{2}, \frac{\ell + 1}{2}),
		\end{align*}
		completing the proof of Theorem 3.
	\end{proof}

\noindent\textbf{Acknowledgment.}
		We thank Susan Jane Colley and two anonymous referees, as well as the participants of Dagstuhl Seminar 16232, for their very helpful comments that considerably improved this article.
	
	%

\begin{thebibliography}{1}
		
		
		\bibitem{Anbarci15} N. Anbarci, C.-J. Sun, M. U. \"{U}nver (2015). Designing fair tiebreak mechanisms: The case of FIFA penalty shootouts. http://ssrn.com/abstract=2558979
		
		\bibitem{anderson1977} C. L. Anderson (1977). Note on the advantage of first serve. \textit{Journal of Combinatorial Theory, Series A}. 23(3): 363.
		
		\bibitem{barrow2012}
		J. D. Barrow (2012).
		\newblock {\em Mathletics: 100 Amazing Things You Didn't Know about the World
			of Sports}.
		\newblock W. W. Norton.
		
		\bibitem{brams1999}
		S. J. Brams, A. D. Taylor (1999).
		\newblock {\em The Win-Win Solution: Guaranteeing Fair Shares to Everybody}.
		\newblock W. W. Norton.
		
		\bibitem{brams2016}
		S. J. Brams, M. S. Ismail (2018). Making the rules of sports fairer. \textit{SIAM Review}. 60(1): 181--202.
		
		\bibitem{cohen-zada2018}
		D. Cohen-Zada, A. Krumer, O. M. Shapir (2018).
		\newblock Testing the effect of serve order in tennis tiebreak.
		\newblock {\em Journal of Economic Behavior \& Organization}. 146: 106--115.
		
		\bibitem{hess2014}
		D. Hess (2014).
		\newblock {\em Golf on the Moon: Entertaining Mathematical Paradoxes and Puzzles}.
		\newblock Mineola, NY: Dover Publications.
		
		\bibitem{isaksen2015}
		A. Isaksen, M. Ismail, S. J. Brams, A. Nealen (2015).
		\newblock Catch-{U}p: A game in which the lead alternates.
		\newblock {\em Game \& Puzzle Design}. 1(2): 38--49.
		
		\bibitem{kemeny1983}
		J. G. Kemeny, J. L. Snell (1983).
		\newblock {\em Finite Markov Chains}.
		\newblock New York: Springer.
		
		\bibitem{kingston1976}
		J. G. Kingston (1976).
		\newblock Comparison of scoring systems in two-sided competitions.
		\newblock {\em Journal of Combinatorial Theory, Series A}. 20(3): 357--362.
		
		\bibitem{palacios2014}
		I. Palacios-Huerta (2014).
		\newblock {\em Beautiful Game Theory: How Soccer Can Help Economics}.
		\newblock Princeton Univ. Press.
		
		\bibitem{pauly2014}
		M. Pauly (2014).
		\newblock Can strategizing in round-robin subtournaments be avoided?
		\newblock {\em Social Choice and Welfare}. 43(1): 29--46.
		
		\bibitem{ruffle2015}
		B. J. Ruffle, O. Volij (2016).
		\newblock First-mover advantage in best-of series: an experimental comparison
		of role-assignment rules.
		\newblock {\em International Journal of Game Theory}. 45(4): 933--970.
		
		\bibitem{schilling2009}
		M. F. Schilling (2009).
		\newblock Does momentum exist in competitive volleyball?
		\newblock {\em CHANCE}. 22(4): 29--35.
		
		\bibitem{neumann1944}
		J. von Neumann, O. Morgenstern (1953).
		\newblock {\em Theory of Games and Economic Behavior}. Third ed.
		\newblock Princeton Univ. Press.
		
		\bibitem{winkler2018}
		P. Winkler (2018).
		\newblock Probability in Your Head.
		\newblock In: J. Beineke \& J. Rosenhouse, eds.
		\newblock {\em The Mathematics of Various Entertaining Subjects (MOVES) III}.
		\newblock National Museum of Mathematics \& Princeton Univ. Press, to appear.
		
	\end{thebibliography}

\vfill\eject

\end{document}